\documentclass{amsart}
\usepackage{setspace, amsmath, amsthm, amssymb, amsfonts, amscd, epic, graphicx, ulem, dsfont}
\usepackage[T1]{fontenc}
\usepackage{multirow}
\usepackage{bbm}

\newtheorem{lem}{Lemma}
\newtheorem{thm}{Theorem}
\newtheorem*{un-thm}{Theorem}

\newtheorem{pro}{Proposition}
\newtheorem{cor}{Corollary}

\theoremstyle{abstract}

\theoremstyle{remark}
\newtheorem*{rema}{Remark}

\theoremstyle{definition}

\newcommand{\R}{\mathbb{R}}

\makeatletter \@namedef{subjclassname@2010}{
  \textup{2010} Mathematics Subject Classification}
\makeatother

\begin{document}

\title[The Sum of Two Unbounded Linear Operators]{The Sum of Two Unbounded Linear Operators: Closedness, Self-adjointness and Normality}
\author{Mohammed Hichem Mortad}

\address{Department of
Mathematics, University of Oran, B.P. 1524, El Menouar, Oran 31000.
Algeria.\newline {\bf Mailing address}:
\newline Dr Mohammed Hichem Mortad \newline BP 7085
Es-Seddikia\newline Oran
\newline 31013 \newline Algeria}

\email{mhmortad@gmail.com, mortad@univ-oran.dz.}

\begin{abstract}
In the present paper we give results on the closedness and the
self-adjointness of the sum of two unbounded operators. We present a
new approach to these fundamental questions in operator theory. We
also prove a new version of the Fuglede theorem where the operators
involved are all unbounded. This new Fuglede theorem allows us to
prove (under extra conditions) that the sum of two unbounded normal
operators remains normal. Also a result on the normality of the
unbounded product of two normal operators is obtained as a
consequence of this new "Fuglede theorem". Some interesting examples
are also given.
\end{abstract}



\thanks{Partially supported by "Laboratoire d'Analyse Mathématique et Applications".}

\maketitle

\section{Introduction}

We start with some standard notions and results about linear
operators on a Hilbert space. We assume the reader is familiar with
other results and definitions about linear operators. Some general
references are \cite{GGK,Gold-Book-1966,Kato-Book,MV,Putnam-book,
RUD,WEI}.

All operators are assumed to be densely defined together with any
operation involving them or their adjoints. Bounded operators are
assumed to be defined on the whole Hilbert space.

If $A$ and $B$ are two unbounded operators with domains $D(A)$ and
$D(B)$ respectively, then $B$ is called an extension of $A$, and we
write $A\subset B$, if $D(A)\subset D(B)$ and if $A$ and $B$
coincide on $D(A)$.   If $A\subset B$, then $B^*\subset A^*$.

The product $AB$ of two unbounded operators $A$ and $B$ is defined
by
\[BA(x)=B(Ax) \text{ for } x\in D(BA)\]
where
\[D(BA)=\{x\in D(A):~Ax\in D(B)\}.\]

Since the expression $AB=BA$ will be often met, and in order to
avoid possible confusions, we recall that by writing $AB=BA$, we
mean that $ABx=BAx$ for all $x\in D(AB)=D(BA)$.

Recall that the unbounded operator $A$, defined on a Hilbert space
$H$, is said to be invertible if there exists an \textit{everywhere
defined} (i.e. on the whole of $H$) bounded operator $B$ such that
\[BA\subset AB=I\]
where $I$ is the usual identity operator. This is the definition
adopted in the present paper. It may be found in e.g. \cite{Con} or
\cite{GGK}. We insist on the inverse being defined everywhere since
if it were not, that it is known from the literature that some of
the results to be proved (for instance Corollary \ref{AB normal A B
Invertible}) may fail to hold. Of course, in some textbooks, they do
not assume the inverse defined everywhere as in e.g.
\cite{Kato-Book}.

An unbounded operator $A$ is said to be closed if its graph is
closed; self-adjoint if $A=A^*$ (hence from known facts self-adjoint
operators are automatically closed); normal if it is \textit{closed}
and $AA^*=A^*A$ (this implies that $D(AA^*)=D(A^*A)$).

The following lemma is standard (for a proof see eg \cite{GGK})

\begin{lem}\label{invetrible closed}
Let $A$ be a densely defined operator. If $A$ is invertible, then it
is closed.
\end{lem}

When dealing with self-adjoint and normal operators, taking adjoints
is compulsory, so we list some straightforward results about the
adjoint of the sum and the product of unbounded operators.

\begin{thm}\label{adjoints and closedness A+B AB basic}
Let $A$ be an unbounded operator.
\begin{enumerate}
  \item $(A+B)^*=A^*+B^*$ if $B$ is bounded, and $(A+B)^*\supset A^*+B^*$ if $B$ is unbounded.
  \item $A+B$ is closed if $A$ is assumed to be closed and if $B$ is bounded.
  \item $(BA)^*=A^*B^*$ if $B$ is bounded.
  \item $A^*B^*\subset (BA)^*$ for any unbounded $B$ and if $BA$ is
  densely defined.
\end{enumerate}
\end{thm}

The following is also well-known
\begin{lem}\label{AB closed}The product $AB$ (in this order) of
two closed operators is closed if one of the following occurs:
\begin{enumerate}
  \item $A$ is invertible,
  \item $B$ is bounded.
\end{enumerate}
\end{lem}

The following lemma is essential

\begin{lem}[\cite{WEI}]\label{(AB)*=B*A*} If $A$ and $B$ are densely defined and $A$ is invertible with
inverse $A^{-1}$ in $B(H)$, then $(BA)^* =A^* B^*$.
\end{lem}

We include a proof (not outlined in \cite{WEI}) to show the
importance of assuming the inverse defined everywhere:

\begin{proof}
Since $A$ is invertible, $AA^{-1}=I$ (in \cite{Kato-Book} we would
have $AA^{-1}\subset I$ only) and hence
\[BAA^{-1}=B\Longrightarrow (A^{-1})^*(BA)^*\subset[(BA)A^{-1}]^*\subset B^*.\]
But $(A^{-1})^*=(A^*)^{-1}$ and so
\[A^*(A^*)^{-1}(BA)^*=(BA)^*\subset A^*B^*.\]
Since we always have $A^*B^*\subset (BA)^*$ the result follows.
\end{proof}

It is worth noticing that similar papers on sums and products exist.
The interested reader may look at \cite{vandaele}, \cite{Gesztesy},
\cite{Gust-CMB-2011}, \cite{HK}, \cite{Jorgensen-PALLE}, \cite{KOS},
\cite{Mess-Mortad-Djell-Azz}, \cite{Mortad-PAMS2003},
\cite{Mortad-IEOT-2009}, \cite{Mortad-CMB-2011},
\cite{mortad-CAOT-sum-normal}, \cite{Mortad-Demm-math},
\cite{Mortad-product-unbounded-kaplansky}, \cite{Mortad-Madani},
\cite{Seb-Stochel} and \cite{Vasilescu-1983-anti-commuting}, and
further bibliography cited therein.

Let us briefly say a few words on how the paper is organized. In the
main results section, we start by proving the first result on the
closedness of the sum. Then a self-adjointness result is
established. Also, it is proved that the adjoint of the sum is the
sum of the adjoints. To treat the case of the normality of the sum
of two normal operators, we prove a new version of the Fuglede
theorem where the operators involved are unbounded. This last result
can too be used to prove a result on the normality of the product of
two unbounded normal operators.

\section{Main Results}

We start by the closedness of the sum. We have

\begin{thm}\label{closed A+B}
Let $A$ and $B$ be two unbounded operators such that $AB=BA$. If $A$
(for instance) is invertible, $B$ is  closed and $D(BA^{-1})\subset
D(A)$, then $A+B$ is closed on $D(B)$.
\end{thm}

\begin{proof}
First note that the domain of $A+B$ is $D(A)\cap D(B)$. But
\[AB=BA\Longrightarrow A^{-1}B\subset BA^{-1}\Longrightarrow D(B)=D(A^{-1}B)\subset D(BA^{-1})\subset D(A),\]
hence the domain of $A+B$ is $D(B)$.

By Lemma \ref{invetrible closed}, $A$ is automatically closed. Then
we have
\begin{align*}
A+B&=A+BAA^{-1}\\
&=A+ABA^{-1}\\
&=A(I+BA^{-1}) \text{ (since $D(BA^{-1})\subset D(A)$)}.
\end{align*}
Since $A^{-1}$ is bounded (and $B$ is closed), $BA^{-1}$ is closed,
hence by Theorem \ref{adjoints and closedness A+B AB basic}
$I+BA^{-1}$ is closed so that $A(I+BA^{-1})$ is also closed by Lemma
\ref{AB closed}, proving the closedness of $A+B$ on $D(B)$.
\end{proof}

\begin{rema}
Before going further in this paper and since conditions of the type
$D(BA^{-1})\subset D(A)$ will be often met, we give an example of a
couple of two unbounded operators satisfying this latter condition.
Let $A$ and $B$ be the two unbounded closed operators defined by
\[Af(x)=(x^2+1)^2f(x) \text{ and } Bf(x)=(x^2+1)f(x)\]
on their respective domains
\[D(A)=\{f\in L^2(\R):~(x^2+1)^2f\in L^2(\R)\}\]
\text{ and } \[D(B)=\{f\in L^2(\R):~(x^2+1)f\in L^2(\R)\}.\] The
operator $B$ is invertible with a bounded inverse given by
\[B^{-1}f(x)=\frac{1}{1+x^2}f(x)\]
on the whole Hilbert space $L^2(\R)$. Then
\[D(AB^{-1})=\{f\in L^2(\R):~\frac{1}{1+x^2}f\in L^2(\R),\frac{(1+x^2)^2}{1+x^2}f=(1+x^2)f\in L^2(\R)\}=D(B).\]
\end{rema}

\begin{rema}
The hypothesis $D(BA^{-1})\subset D(A)$ cannot merely be dropped. As
a counterexample, let $A$ be any invertible closed operator with
domain $D(A)\subsetneqq H$ where $H$ is a complex Hilbert space. Let
$B=-A$. Then $A+B=0$ on $D(A)$ is not closed. Moreover, $AB=BA$ but
\[D(BA^{-1})=D(AA^{-1})=D(I)=H\not\subset D(A).\]
\end{rema}

We now pass to the self-adjointness of the sum. We have

\begin{thm}[cf. Lemma 4.15.1 in \cite{Putnam-book}]\label{self-adjoint A+B}
Let $A$ and $B$ be two unbounded self-adjoint operators such that
$B$ (for instance) is also invertible. If $AB=BA$ and
$D(AB^{-1})\subset D(B)$, then $A+B$ is self-adjoint on $D(A)$.
\end{thm}

\begin{proof}
It is clear, thanks to $D(A)=D(B^{-1}A)\subset D(AB^{-1})\subset
D(B)$, that the domain of $A+B$ is $D(A)$. We have
\[B^{-1}\underbrace{BA}_{=AB}+B\subset A+B\]
and hence
\[(I+B^{-1}A)B\subset A+B.\]
So we have
\begin{align*}
(A+B)^*\subset & [(I+B^{-1}A)B]^*\\
=& B^*(I+B^{-1}A)^* \text{ (by Lemma \ref{(AB)*=B*A*} since $B$ is invertible)}\\
=& B^*[I+(B^{-1}A)^*] \text{ (by Theorem \ref{adjoints and
closedness
A+B AB basic})}\\
=&B^*[I+A^*(B^{-1})^*] \text{ (since $B^{-1}$ is bounded)}\\
=&B(I+AB^{-1}) \text{ (since $A$ and $B$ are self-adjoint)}\\
=&B+BAB^{-1} \text{ (for $D(AB^{-1})\subset D(B)$)}\\
=&B+ABB^{-1} \\
=&A+B.
\end{align*}
The following known fact
\[A+B\subset (A+B)^*,\]
then makes the "inclusion" an exact equality, establishing the
self-adjointness of $A+B$ on $D(A)$.
\end{proof}

\begin{rema}
The condition $D(AB^{-1})\subset D(A)$ cannot just be dispensed
with. As before, take $A$ to be any unbounded and invertible
self-adjoint operator and $B=-A$.
\end{rema}

\begin{rema}
Of course, writing $AB=BA$ does not mean that $A$ and $B$ strongly
commute, i.e. it does not signify that their spectral projections
commute. See e.g. \cite{DevNussbaum-von-Neumann}, \cite{FUG-19825},
\cite{Nelson-Anal-vectors}, \cite{RS1} and
\cite{Schmudgen-Friedrich-II}.

However, a result by Devinatz-Nussbaum-von Neumann (see
\cite{DevNussbaum-von-Neumann}) shows that if there exists a
self-adjoint operator $T$ such that $T\subseteq T_1T_2$, where $T_1$
and $T_2$ are self-adjoint, then $T_1$ and $T_2$ strongly commute.
Thus we have

\begin{pro}
Let $A$ and $B$ be two unbounded self-adjoint operators such that
$B$ (for instance) is also invertible. If $AB=BA$, then $A$ and $B$
strongly commute.
\end{pro}

\begin{proof}
By Lemma \ref{(AB)*=B*A*}, we may write
\[(AB)^*=B^*A^*=BA=AB,\]
i.e. $AB$ is self-adjoint. By the Devinatz-Nussbaum-von Neumann
theorem, $A$ and $B$ strongly commute.
\end{proof}
\end{rema}

We can also give a result on the adjoint of the sum of two closed
operators. This generalizes the previous one as we will be
explaining in a remark below its proof. Besides it will be useful in
the case of the sum of two normal operators. We have

\begin{thm}\label{adjoint A+B}
Let $A$ and $B$ be two unbounded invertible operators such that
$AB=BA$. If $D(A^*(B^*)^{-1})\subset D(B^*)$, then
\[(A+B)^*=A^*+B^*.\]
\end{thm}

\begin{proof}
The idea of proof is akin to that of Theorem \ref{self-adjoint A+B}.
First, we always have
\[A^*+B^*\subset (A+B)^*.\]
Second, since $A$ and $B$ are both invertible, by Lemma
\ref{(AB)*=B*A*}, we have
\[AB=BA\Longrightarrow A^*B^*=B^*A^*.\]
Now write
\[B^{-1}\underbrace{BA}_{=AB}+B\subset A+B\]
so that
\[(I+B^{-1}A)B\subset A+B\]
and hence

\begin{align*}
(A+B)^*\subset & [(I+B^{-1}A)B]^*\\
=& B^*(I+B^{-1}A)^* \text{ (by Lemma \ref{(AB)*=B*A*} since $B$ is invertible)}\\
=& B^*[I+(B^{-1}A)^*] \text{ (by Theorem \ref{adjoints and
closedness
A+B AB basic})}\\
=&B^*[I+A^*(B^{-1})^*] \text{ (since $B^{-1}$ is bounded)}\\
=&B^*+B^*A^*(B^*)^{-1} \text{ (as $D(A^*(B^*)^{-1})\subset D(B^*)$)}\\
=&B^*+A^*B^*(B^*)^{-1} \text{ (for $A^*B^*=B^*A^*$)}\\
=&A^*+B^*.
\end{align*}
The proof is therefore complete.
\end{proof}

\begin{rema}
We could have supposed that only $B$ is invertible, but then we
would have added the hypothesis $B^*A^*\subset A^*B^*$. This latter
observation tells us that Theorem \ref{adjoint A+B} generalizes in
fact Theorem \ref{self-adjoint A+B}.
\end{rema}

\begin{rema}
The condition $D(A^*(B^*)^{-1})\subset D(B^*)$ cannot just be
dispensed with. As before, take $A$ to be any unbounded closed and
invertible operator and $B=-A$. Then $D(A^*(B^*)^{-1})\subset
D(B^*)$ is not satisfied. At the same time observe that
\[D(A^*-A^*)=D(A^*)\neq D((A-A)^*)=D(0^*)=H,\]
where $H$ is the whole Hilbert space.
\end{rema}

In \cite{mortad-CAOT-sum-normal}, we proved the following result
\begin{thm}\label{normal A+B, Unbd A, Unbd B-CAO-2012}
Let $A$ and $B$ be two unbounded normal operators such that $B$ is
$A$-bounded with relative bound smaller than one.  Assume that
$BA^*\subset A^*B$ and $B^*A\subset AB^*$.  Then $A+B$ is normal on
$D(A)$.
\end{thm}

To prove it, we had to use a theorem by Hess-Kato (see \cite{HK}),
mainly for the closedness of $A+B$ and to have $(A+B)^*=A^*+B^*$.
Thanks to Theorems \ref{closed A+B} \& \ref{adjoint A+B} we may
avoid the use of that theorem. Besides we are able here to prove a
new version of the Fuglede theorem where all operators involved are
\textit{unbounded} which will allow us to establish the normality of
the sum of two normal operators. We digress a bit to say that
another all unbounded-operator-version of Fuglede-Putnam is the
Fuglede-Putnam-Mortad theorem that may be found in
\cite{Mortad-Fuglede-Putnam-CAOT-2011}, cf. \cite{FUG,PUT}.

Here is the promised result
\begin{thm}\label{Fuglede-NEW-ALL-UNBD}
Let $A$ and $B$ be two unbounded normal and invertible operators.
Then
\[AB=BA\Longrightarrow AB^*=B^*A \text{ and } BA^*=A^*B.\]
\end{thm}

\begin{proof}
Since $B$ is invertible, we may write
\[AB=BA \Longrightarrow B^{-1}A\subset AB^{-1}.\]
Since $B^{-1}$ is bounded and $A$ is unbounded and normal, by the
classic Fuglede theorem we have
\[B^{-1}A\subset AB^{-1}\Longrightarrow B^{-1}A^*\subset A^*B^{-1}.\]
Therefore,
\[A^*B\subset BA^*.\]
But $B$ is invertible, then by Lemma \ref{(AB)*=B*A*} we may obtain
\[AB^*\subset(BA^*)^*\subset (A^*B)^*=B^*A.\]

Interchanging the roles of $A$ and $B$, we shall get
\[B^*A\subset AB^* \text{ and } BA^*\subset A^*B.\]

Thus
\[B^*A=AB^* \text{ and } BA^*=A^*B.\]
\end{proof}

\begin{rema}
A similar result holds with one operator assumed normal. The key
point again is that the inverse is bounded and everywhere defined.
So since $AB=BA$, we obtain $A^{-1}B^{-1}\subset B^{-1}A^{-1}$.
Since these operators are everywhere defined, we get
$A^{-1}B^{-1}=B^{-1}A^{-1}$. The rest follows by the bounded version
of Fuglede theorem.

However, if we do not assume the bounded inverse defined everywhere,
then $AB=BA$ does not imply that $A^{-1}B^{-1}=B^{-1}A^{-1}$. Here
is a counterexample which appeared in \cite{Schmudgen-Friedrich-II}.
It reads:

Let $S$ be the unilateral shift on the Hilbert space $\ell^{2}$. We
may then easily show that both $S+S^*$ and $S-S^*$ are injective.
Hence $A=(S+S^*)^{-1}$ and  $B=i(S-S^*)^{-1}$ are unbounded
self-adjoint such that $AB=BA$. Nonetheless
\[A^{-1}B^{-1}\neq B^{-1}A^{-1}\] since otherwise $S$ and $S^*$ would
commute!
\end{rema}

As a first consequence of Theorem \ref{Fuglede-NEW-ALL-UNBD}, we
have

\begin{thm}\label{normal A+B NEW}
Let $A$ and $B$ be two unbounded invertible normal operators with
domains $D(A)$ and $D(B)$ respectively. If $AB=BA$, $D(A)\subset
D(BA^*)$ and $D(AB^{-1})=D(A(B^*)^{-1})\subset D(B)$, then $A+B$ is
normal on $D(A)$.
\end{thm}

\begin{proof}
To prove that $A+B$ is normal, we must say why $A+B$ is closed and
show that
\[(A+B)^*(A+B)=(A+B)(A+B)^*.\]
$A+B$ is closed thanks to Theorem \ref{closed A+B}. Also, since $A$
and $B$ are both normal, we obviously have
\[D(A(B^*)^{-1})\subset D(B)\Longrightarrow D(A^*(B^*)^{-1})\subset D(B^*)\]
so that Theorem \ref{self-adjoint A+B} applies and yields
\[(A+B)^*=A^*+B^*.\]

Since $AB=BA$, Theorem \ref{Fuglede-NEW-ALL-UNBD} implies that
\[AB^*=B^*A \text{ and } BA^*=A^*B.\]

Since $D(A)\subset D(BA^*)$, we have
\[D(A)\subset D(A^*B)~(\subset D(B))\]
\text{ and }
\[D(A)\subset D(BA^*)=D(A^*B)=D(AB)=D(BA)=D(B^*A).\]

All these domain inclusions allow us to have

\begin{enumerate}
  \item $A^*(A+B)=A^*A+A^*B$ and $B^*(A+B)=B^*A+B^*B$.
  \item $A(A^*+B^*)=AA^*+AB^*$ and $B(A^*+B^*)=BA^*+BB^*$.
\end{enumerate}

Hence we may write
\[
(A+B)^*(A+B)=A^*A+A^*B+B^*A+B^*B
\]
and
\[(A+B)(A+B)^*=AA^*+AB^*+BA^*+BB^*.\]

Thus

\[(A+B)^*(A+B)=(A+B)(A+B)^*.\]
\end{proof}

As another consequence of Theorem \ref{Fuglede-NEW-ALL-UNBD}, we
have the following result on the normality of the product of two
unbounded normal operators.

\begin{cor}\label{AB normal A B Invertible}
Let $A$ and $B$ be two unbounded invertible normal operators. If
$BA=AB$, then $BA$ (and $AB$) is normal.
\end{cor}

\begin{rema}
By a result of Devinatz-Nussbaum (see \cite{DevNussbaum}), if $A$,
$B$ and $N$ are normal where $N=AB=BA$, then $A$ and $B$ strongly
commute. Hence we have as a consequence:

\begin{cor}
Let $A$ and $B$ be two unbounded invertible normal operators. If
$BA=AB$, then $A$ and $B$ strongly commute.
\end{cor}

\end{rema}

Now we prove Corollary \ref{AB normal A B Invertible}:

\begin{proof}
We first note that $BA$ is closed thanks to Lemma \ref{AB closed}
(or simply since $BA$ is invertible hence Lemma \ref{invetrible
closed} applies!), hence so is $AB$. Lemma \ref{(AB)*=B*A*} then
gives us $(AB)^*=B^*A^*$. By Theorem \ref{Fuglede-NEW-ALL-UNBD} we
may then write
\begin{align*}
(AB)^*AB&=B^*A^*AB\\
&=B^*AA^*B\\
&=AB^*A^*B\\
&=AB^*BA^*\\
&=ABB^*A^*\\
&=(AB)(AB)^*,\\
\end{align*}
establishing the normality of $AB$.
\end{proof}

\begin{rema}
In \cite{Mortad-Madani} we had the same result with the extra
conditions $D(A),D(B)\subset D(BA)$. Here we have showed that the
last two conditions are not essential. Hence Corollary \ref{AB
normal A B Invertible} is an improvement of the result that appeared
in \cite{Mortad-Madani}.
\end{rema}

\begin{rema}
Let us give an example that shows the importance of assuming $A$ and
$B$ invertible. Let $B$ be the operator defined by
\[Bf(x)=-xf'(x)-f(x)\]
on its domain
\[D(B)=\{f\in L^2(\R):~xf'\in
L^2(\R)\}\] where the derivative is taken in the distributional
sense. Then $B$ is normal (it is in fact the adjoint of the operator
defined by $xf'(x)$ on the domain $\{f\in L^2(\R):~xf'\in
L^2(\R)\}$). Set $A=B+I$ hence
\[Af(x)=-xf'(x)\]
on
\[D(A)=\{f\in L^2(\R):~xf'\in L^2(\R)\}.\]
We may then easily check that
\[ABf(x)=BAf(x)=x^2f''(x)\]
on their common domain
\[D(B^2)=\{f\in L^2(\R):~xf',x^2f''\in
L^2(\R)\}.\] Hence $AB$ and $BA$ are not closed, hence they are not
normal.

Now, proceeding as in \cite{Mortad-IEOT-2009} (where a similar
operator was dealt with) we may show that via a form of the Mellin
transform that $B$ is unitary equivalent to the multiplication
operator $M$ defined by
\[Mf(x)=(x-\frac{1}{2}i)f(x)\]
on its domain
\[D(M)=\{f\in L^2(\R):~xf\in L^2(\R)\}.\]
But $M$ is known to be non invertible, so neither is $B$ nor is $A$.
\end{rema}

\section{Conclusion}

Lemma \ref{(AB)*=B*A*} has played a very important role in the
proofs of most of the results in the present paper. Of course, we
could have used other similar known results in the literature. See
for instance \cite{Castr-Gold} and \cite{Sch-1970}. Theorem
\ref{Fuglede-NEW-ALL-UNBD} also played an important role in the
proof of Theorem \ref{normal A+B NEW}. We could have also used the
Fuglede-Putnam-Mortad theorem which appeared in
\cite{Mortad-Fuglede-Putnam-CAOT-2011}. We also think that Theorem
\ref{Fuglede-NEW-ALL-UNBD} should have other applications somewhere
else.

\end{document}